\documentclass[12pt,a4paper]{amsart}
\usepackage[all]{xy}
\usepackage{enumerate}
\usepackage{amsmath,amssymb,xspace,amsthm}

\newtheorem{theorem}{Theorem}
\newtheorem{lemma}[theorem]{Lemma}
\newtheorem{proposition}[theorem]{Proposition}
\numberwithin{equation}{section}

\begin{document}

\title[Simple Virasoro modules]{Simple Virasoro modules induced from codimension one subalgebras of 
the positive part}
\author{Volodymyr Mazorchuk and Emilie Wiesner}
\date{\today}

\begin{abstract}
We construct a new five-parameter family of simple modules over the Virasoro Lie algebra.
\end{abstract}

\maketitle

%
%

\section{Introduction and description of the results}\label{s1}

We denote by $\mathbb{N}$ the set of positive integers and by $\mathbb{Z}_+$ the set of all  non-negative integers. 
For a Lie algebra $\mathfrak{a}$ we denote by $U(\mathfrak{a})$ the universal enveloping algebra of $\mathfrak{a}$.

Let $\mathfrak{V}$ denote the complex {\em Virasoro algebra}, that is the Lie algebra with basis $\{\mathtt{c},
\mathtt{l}_i:i\in\mathbb{Z}\}$ and the Lie bracket defined (for $i,j\in\mathbb{Z}$) as follows: 
\begin{displaymath}
[\mathtt{l}_i,\mathtt{l}_j]=(j-i) \mathtt{l}_{i+j}+\delta_{i,-j}\frac{i^3-i}{12}\mathtt{c};\quad
[\mathtt{l}_i,\mathtt{c}]=0.
\end{displaymath}
This algebra plays an important role in various questions of mathematical physics, see \cite{KR}.

Classical classes of simple weight $\mathfrak{V}$-modules are simple highest weight modules, see \cite{FF}, 
and intermediate series modules. Put together, these two classes exhaust all simple weight $\mathfrak{V}$-modules with 
finite dimensional weight spaces, see \cite{Mt}, and even those containing a finite dimensional weight space, see 
\cite{MZ1}. We also refer the reader to the recent monograph \cite{IK} for a detailed survey of the classical part of
the representation theory of $\mathfrak{V}$. There are a number of other examples of simple $\mathfrak{V}$-modules
constructed in \cite{Zh,OW,GWZ,LGZ,MZ2} using various tricks. The present note contributes with a new trick
leading to a new family of examples depending on five parameters.

For a nonzero $z\in\mathbb{C}$, denote by $\mathfrak{a}_z$ the linear span of $\mathtt{l}_k-z^{k-1}\mathtt{l}_1$, 
where $k\geq 2$, in $\mathfrak{V}$. It is easy to check that $\mathfrak{a}_z$ is in fact a Lie subalgebra 
in $\mathfrak{V}$. For a fixed $\mathbf{m}:=(m_2,m_3,m_4)\in\mathbb{C}^3$, define an $\mathfrak{a}_z$-action 
on $\mathbb{C}$ by 
\begin{itemize}
\item $(\mathtt{l}_i-z^{i-1}\mathtt{l}_1)\cdot 1= m_i$ for $i=2,3,4$;
\item $(\mathtt{l}_i-z^{i-1}\mathtt{l}_1)\cdot 1=-(i-4)m_3z^{i-3}+(i-3)m_4z^{i-4}$ for $i>4$.
\end{itemize}
It's straightforward to verify that this gives an $\mathfrak{a}_z$-module. We denote it by $\mathbb{C}_{\mathbf{m}}$.

For a fixed $\theta\in\mathbb{C}$ consider the  $\mathfrak{V}$-module
\begin{displaymath}
\mathrm{Ind}_{z,\theta}(\mathbb{C}_{\mathbf{m}}):=
U(\mathfrak{V})\otimes_{U(\mathfrak{a}_z)}\mathbb{C}_{\mathbf{m}}/
(\mathtt{c}-\theta)U(\mathfrak{V})\otimes_{U(\mathfrak{a}_z)}\mathbb{C}_{\mathbf{m}}.
\end{displaymath}
Our main result is the following claim.

\begin{theorem}\label{themmain}
Let $\mathbf{m}\in\mathbb{C}^3$ and $z\in\mathbb{C}\setminus\{0\}$ be such that
\begin{equation}\label{conditions}
zm_3\neq m_4, \, \, 2zm_2\neq m_3, \, \, 3 zm_3\neq 2m_4, \,\text{ and } \, z^2m_2+m_4\neq 2zm_3.
\end{equation}
Then for any $\theta\in\mathbb{C}$ the $\mathfrak{V}$-module
$\mathrm{Ind}_{z,\theta}(\mathbb{C}_{\mathbf{m}})$ is simple.
\end{theorem}

We start the paper by characterizing $\mathfrak{a}_z$ as a special family of codimension one 
subalgebras in the positive part $\mathfrak{n}$ of $\mathfrak{V}$ in Section~\ref{s2}. Our proof of 
Theorem~\ref{themmain} consists of three steps: we first induce $\mathbb{C}_{\mathbf{m}}$ up to $\mathfrak{n}$ in
Section~\ref{s3}, then we study the induction of $\mathbb{C}_{\mathbf{m}}$ to the Borel subalgebra 
of $\mathfrak{V}$ in Section~\ref{s4} and, finally, we complete the proof in Section~\ref{s5}.
\vspace{2mm}

\noindent

{\bf Acknowledgement.} The major part of this work was done during the visit of the second author to Uppsala
in May 2012. The hospitality and financial support of Uppsala University are gratefully acknowledged. The first 
author is partially supported by the Royal Swedish Academy of Sciences and the Swedish Research Council.

\section{Subalgebras in $\mathfrak{n}$ of codimension one}\label{s2}

Denote by $\mathfrak{n}$ the ``positive part'' of $\mathfrak{V}$, that is the Lie subalgebra spanned by
$\mathtt{l}_i$, where $i\geq 1$. In this section we characterize $\mathfrak{a}_z$ as a special family of 
codimension one  subalgebras in $\mathfrak{n}$. We start with the following lemma.

\begin{lemma}\label{lem1}
Let $\mathfrak{x} \subseteq \mathfrak{n}$ be a Lie subalgebra of codimension one.   
\begin{enumerate}[$($a$)$]
\item\label{lem1.1} If $\mathtt{l}_1 \in \mathfrak{x}$, then $\mathtt{l}_k, \mathtt{l}_{k+1}, 
\ldots \in \mathfrak{x}$ for some $k\in\mathbb{N}$.
\item\label{lem1.2} If $\mathtt{l}_s \in \mathfrak{x}$ for some $s\in\mathbb{N}$, then 
$\mathtt{l}_{is}, \mathtt{l}_{(i+1)s}, \mathtt{l}_{(i+2)s},  \ldots \in \mathfrak{x}$ for some $i\in\mathbb{N}$.
\item\label{lem1.3} If $\mathtt{l}_{is},\mathtt{l}_{(i+1)s}, \mathtt{l}_{(i+2)s},  \ldots \in \mathfrak{x}$ 
for some $s,i\in\mathbb{N}$, then  $\mathtt{l}_k, \mathtt{l}_{k+1}, \ldots \in \mathfrak{x}$ for some $k$.
\end{enumerate}
\end{lemma}

\begin{proof}
If $\mathtt{l}_k \in \mathfrak{a}$ for some $k>1$, then claim \eqref{lem1.1} follows immediately. 
Therefore, assume that this is not the case and in particular that claim \eqref{lem1.1} is false.  Since $\mathfrak{x}$ 
has codimension $1$ and $\mathtt{l}_k \not\in \mathfrak{x}$ for any $k>1$, we must have that 
$\mathtt{l}_k-a_k\mathtt{l}_{k+1} \in \mathfrak{x}$ for all $k\geq 2$ and some $0 \neq a_k \in \mathbb{C}$.  
Also, $[\mathtt{l}_1, \mathtt{l}_k-a_k\mathtt{l}_{k+1}] \in \mathfrak{x}$, which implies $a_{k+1}=\frac{k}{k-1}a_{k}=ka_2$.  On the other hand, 
\begin{eqnarray*}
[\mathtt{l}_2-a_2\mathtt{l}_3,\mathtt{l}_3-a_3\mathtt{l}_4]&=&\mathtt{l}_5-2a_3\mathtt{l}_6+a_2a_3\mathtt{l}_7\\
&=&\mathtt{l}_5-4a_2\mathtt{l}_6+2a_2^2\mathtt{l}_7.
\end{eqnarray*}
Since $\mathtt{l}_5-a_5\mathtt{l}_6=\mathtt{l}_5-4a_2\mathtt{l}_6$, this implies that 
$\mathtt{l}_7 \in \mathfrak{x}$, a contradiction. This proves claim \eqref{lem1.1}. Claim \eqref{lem1.2} follows
from claim \eqref{lem1.1} as the span of all $\mathtt{l}_{js}$, $j\in\mathbb{N}$, is isomorphic 
to $\mathfrak{n}$ as a Lie algebra.

To prove claim \eqref{lem1.3} we show that for any $1\leq p<s$ there is $j\in\mathbb{Z}_+$ such that 
$\mathtt{l}_{p+js}\in \mathfrak{x}$. Assume that this is not the case. Then 
$\mathtt{l}_{p+js}-a_j\mathtt{l}_{p+(j+1)s} \in \mathfrak{x}$ for all $j\in\mathbb{Z}_+$ and some 
$0 \neq a_j \in \mathbb{C}$.  Also, $[\mathtt{l}_{is}, \mathtt{l}_{p+js}-a_j\mathtt{l}_{p+(j+1)s}] 
\in \mathfrak{x}$, which implies $a_{i+j}=\frac{p+(j+1-i)s}{p+(j-i)s}a_j$.  Hence
$a_{2i+j}=\frac{p+(j+1-i)s}{p+(j-i)s}\frac{p+(j+1)s}{p+js}a_j$. 

On the other hand, $[\mathtt{l}_{2is}, \mathtt{l}_{p+js}-a_j\mathtt{l}_{p+(j+1)s}] 
\in \mathfrak{x}$, which implies $a_{2i+j}=\frac{p+(j+1-2i)s}{p+(j-2i)s}a_j$. This gives
\begin{displaymath}
\frac{p+(j+1-i)s}{p+(j-i)s}\cdot \frac{p+(j+1)s}{p+js}= \frac{p+(j+1-2i)s}{p+(j-2i)s}.
\end{displaymath}
for all $j\in\mathbb{Z}_+$. This means that the following multisets
(the negatives of the roots with respect to the variable $js$) coincide:
\begin{displaymath}
\{(1-i)s+p,p+s,p-2is\}=\{p-is,p,p+(1-2i)s\} 
\end{displaymath}
and thus
\begin{displaymath}
\{(1-i)s,s,-2is\}=\{-is,0,(1-2i)s\}.
\end{displaymath}
Since both $i,s\in\mathbb{N}$, the only way to have a zero on the left hand side is to have $i=1$, which implies
\begin{displaymath}
\{0,s,-2s\}=\{-s,0,-s\}.
\end{displaymath}
Now the left hand side contains a positive integer, while the right hand side does not. 
This is a contradiction which proves claim \eqref{lem1.3}.
\end{proof}

\begin{proposition}\label{prop2}
Suppose $\mathfrak{x} \subseteq \mathfrak{n}$ is a subalgebra of codimension one and there is no $k$ such that 
$\mathtt{l}_k, \mathtt{l}_{k+1}, \ldots \in \mathfrak{x}$. Then $\mathfrak{x}$ has a basis of the form
$\{ \mathtt{l}_k-z^{k-1} \mathtt{l}_1\mid k \geq 2 \}$ for some nonzero $z \in\mathbb{C}$.
\end{proposition}

\begin{proof}
By Lemma~\ref{lem1} we have $\mathtt{l}_s \not\in\mathfrak{x}$ for any $s\in \mathbb{N}$. 
As $\mathfrak{x}$ has codimension one in $\mathfrak{n}$, it must have a basis of the form 
$\{\mathtt{l}_k-a_k\mathtt{l}_1 \mid  k \geq 2\}$ for some nonzero $a_k \in \mathbb{C}$.

For $2\leq i<j$ consider 
\begin{multline*}
[\mathtt{l}_i+a_i\mathtt{l}_1, \mathtt{l}_j+a_j\mathtt{l}_1]=\\=
(j-i)\mathtt{l}_{i+j}-(j-1)a_i\mathtt{l}_{j+1}+(i-1)a_j\mathtt{l}_{i+1}=\\
=(j-i)\big(\mathtt{l}_{i+j}-a_{i+j}\mathtt{l}_1\big)
-(j-1)a_i\big(\mathtt{l}_{j+1}-a_{j+1}\mathtt{l}_1\big)
+\\+(i-1)a_j\big(\mathtt{l}_{i+1}-a_{i+1}\mathtt{l}_1\big)+\\+
\big((j-i)a_{i+j}-(j-1)a_ia_{j+1}+(i-1)a_{i+1}a_j\big)\mathtt{l}_1.
\end{multline*}
The condition $[\mathtt{l}_i+a_i\mathtt{l}_1, \mathtt{l}_j+a_j\mathtt{l}_1]\in \mathfrak{x}$ implies
\begin{equation}\label{eq123}
D_{i,j}:=(j-i)a_{i+j}-(j-1)a_ia_{j+1}+(i-1)a_{i+1}a_j=0. 
\end{equation}
Taking $i=2$ and $j=3,4,\dots$ we get a recursive formula which uniquely determines 
$a_5,a_6,\dots$ in terms of $a_2,a_3,a_4$.

Let $I$ be the ideal in $\mathbb{C}[a_2,a_3,a_4,a_5,a_6,a_7,a_8,a_9]$ generated by 
$D_{2,3}$, $D_{2,4}$, $D_{2,5}$, $D_{2,6}$, $D_{2,7}$, $D_{3,4}$, $D_{3,5}$, $D_{3,6}$ and  $D_{4,5}$.
Computing the Gr{\"o}bner basis of $I$ with respect to the lexicographic order for which
\begin{displaymath}
a_9>a_8>a_7>a_6>a_5>a_4>a_3>a_2, 
\end{displaymath}
we get that  $I$ contains $a_3^6-a_3^5a_2^2$, which implies $a_3=a_2^2$ since $a_3\neq 0$.
Computing the Gr{\"o}bner basis of $I$ with respect to the lexicographic order for which
\begin{displaymath}
a_9>a_8>a_7>a_6>a_5>a_2>a_3>a_4, 
\end{displaymath}
we get that  $I$ contains $a_4^3a_2-a_4^2a_3^2$. Using $a_3=a_2^2$ and $a_4,a_2\neq 0$, 
we get $a_4=a_2^3$. This means that all $a_k$ are uniquely determined by the value of $a_2$.

At the same time, it is easy to check that $a_k=a_2^{k-1}$ satisfies \eqref{eq123} and hence defines a subalgebra.
\end{proof}

For a nonzero $z\in\mathbb{C}$ we denote by $\mathfrak{a}_z$ the subalgebra constructed in Proposition ~\ref{prop2}.
This one-parameter family of subalgebras exhausts all codimension one subalgebras of  $\mathfrak{n}$
which do not contain $\mathtt{l}_k, \mathtt{l}_{k+1}, \ldots$ for some $k$. Modules
induced from subalgebras of $\mathfrak{n}$ containing $\mathtt{l}_k, \mathtt{l}_{k+1}, \ldots $ 
for some $k$ were studied in \cite{MZ2}. All such modules fit into the general Whittaker setup for $\mathfrak{V}$
defined in \cite{BM}. Simple modules which are induced from simple $1$-dimensional $\mathfrak{a}_z$-modules do 
not fit into this general Whittaker setup. They are the objects of our study in the present paper.

\section{Induction to $\mathfrak{n}$}\label{s3}

For $\mathbf{m}:=(m_1,m_2,m_3)\in\mathbb{C}^3$ define 
$V_{\mathbf{m}}= \mbox{Ind}_{\mathfrak{a}_z}^{\mathfrak{n}}(\mathbb{C}_{\mathbf{m}})$.  
Since $\mathfrak{n}$ has a basis $\mathtt{l}_1, \mathtt{l}_2-z\mathtt{l}_1, \mathtt{l}_3-z^2\mathtt{l}_1, \ldots$, 
the PBW Theorem implies that $V_{\mathbf{m}}$ has a basis $\{\mathtt{l}_1^k \otimes 1 \mid k \geq 0\}$.

\begin{proposition}\label{prop21}
The $\mathfrak{n}$-module $V_{\mathbf{m}}$ is simple if and only if $am_3 \neq m_4$.
\end{proposition}

\begin{proof}
Suppose $0 \neq M \subseteq V_ {\mathbf{m}}$.  Since $\mathtt{l}_1$ acts freely on $M$, it follows that 
$V_ {\mathbf{m}}/M$ is finite-dimensional. In particular, we may assume that $M$ is chosen so that 
$V_ {\mathbf{m}}/M$ is simple.

Let $I \subseteq \mathfrak{n}$ be the annihilator of this quotient module. In \cite[Subsection~3.3]{MZ2}
it is shown that any ideal in $\mathfrak{n}$ of finite codimension contains $\mathtt{l}_k, \mathtt{l}_{k+1}, \ldots$
for some $k>0$. This implies that $\mathfrak{n}/I$ is nilpotent,   $V_ {\mathbf{m}}/M$ is one-dimensional, 
and $[\mathfrak{n},\mathfrak{n}] \subseteq I$. Therefore, $\mathtt{l}_3$ and $\mathtt{l}_4$ act on  
$V_ {\mathbf{m}}/M$ as zero and thus $z(\mathtt{l}_3-z^2\mathtt{l}_1)$ and $\mathtt{l}_4-z^3\mathtt{l}_1$  act equally,
which yields $zm_3=m_4$.

Assume now that $zm_3=m_4$ and consider the $\mathfrak{n}$-module $\mathbb{C}_{m_2,m_3}$ on which $\mathtt{l}_1$ 
acts via $-\frac{1}{z^2}m_3$ and $\mathtt{l}_2$ acts via $m_2-\frac{1}{z}m_3$. 
Computing its restriction to $\mathfrak{a}_z$ one gets that 
the restriction is isomorphic to $\mathbb{C}_{\mathbf{m}}$. From the universal property of induced modules it 
follows that $V_{\mathbf{m}}$ surjects onto this $\mathfrak{n}$-mo\-dule $\mathbb{C}_{m_2,m_3}$ and hence is reducible.
\end{proof}

From the above proof it follows that in the case $zm_3=m_4$ the module $V_{\mathbf{m}}$ has a unique simple
top isomorphic to the $\mathfrak{n}$-module $\mathbb{C}_{m_2,m_3}$ on which $\mathtt{l}_1$ acts via 
$-\frac{1}{z^2}m_3$ and $\mathtt{l}_2$ acts via $m_2-\frac{1}{z}m_3$. The induced module 
$\mathrm{Ind}_{z,\theta}(\mathbb{C}_{m_2,m_3})$ is completely described in \cite{OW,MZ2}.

For $k>4$ set 
\begin{equation}\label{eq2}
m_k = -(k-4)m_3z^{i-3}+(k-3)m_4z^{i-4}. 
\end{equation}
To simplify our notation, for $k\geq 2$ we will denote by $\hat{\mathtt{l}}_k$ the element
$\mathtt{l}_k-z^{k-1}\mathtt{l}_1\in\mathfrak{a}_z$.
We will need the following property of $V_{\mathbf{m}}$:

\begin{lemma}\label{lemnn1}
Let $k\geq 2$. Then the module $V_{\mathbf{m}}$ has a basis $\{v_n^{(k)}\vert n\in\mathbb{N}_0\}$
such that $\hat{\mathtt{l}}_k\cdot v_n^{(k)}=(m_k+nz^{k}(1-k))v_n^{(k)}$.
\end{lemma}

\begin{proof}
Let $v$ be the canonical generator of $V_{\mathbf{m}}$. For $n\in\mathbb{N}_0$ denote by  $V(n)$ the linear span 
of $v,\mathtt{l}_1v,\dots,\mathtt{l}_1^nv$. We claim that $V(n)$ is invariant under the action of 
$\hat{\mathtt{l}}_k$ and, moreover,
\begin{equation}\label{eq1}
\hat{\mathtt{l}}_k\cdot \mathtt{l}^n_1v=(m_k+nz^{k}(1-k))\mathtt{l}^n_1v \,\,\mathrm{ mod }\,\,V(n-1).
\end{equation}
We show this by induction on $n$. If $n=0$, the claim is clear. To show the induction step, we compute
(by moving $\hat{\mathtt{l}}_k$ through one $\mathtt{l}_1$ and using the inductive assumption in the 
first equality):
{\small
\begin{multline*}
(\mathtt{l}_k-z^{k-1}\mathtt{l}_1)\cdot \mathtt{l}^{n+1}_1v=\\=
(m_k+nz^{k}(1-k))\mathtt{l}^{n+1}_1v + 
[\mathtt{l}_k-z^{k-1}\mathtt{l}_1,\mathtt{l}_1]\mathtt{l}^{n}_1v\,\,\mathrm{ mod }\,\,V(n)=\\=
(m_k+(n+1)z^{k}(1-k))\mathtt{l}^{n+1}_1v +(1-k)(\mathtt{l}_{k+1}-
z^{k}\mathtt{l}_1)\mathtt{l}^{n}_1v\,\,\mathrm{ mod }\,\,V(n).
\end{multline*}
}

This implies \eqref{eq1} and the statement of the lemma follows.
\end{proof}

\section{Induction to the Borel}\label{s4}

Denote by $\mathfrak{b}$ the standard Borel subalgebra of $\mathfrak{V}$, that is the subalgebra generated by
$\mathfrak{n}$ and $\mathtt{l}_0$. Define $W_{\mathbf{m}}= \mbox{Ind}_{\mathfrak{n}}^{\mathfrak{b}}(V_{\mathbf{m}})$.

\begin{proposition} \label{prop:borel}
If $zm_3 \neq m_4$ and $\mathbf{m}\neq (m_2,2zm_2,3z^2m_2)$, then the $\mathfrak{b}$-mo\-dule 
$W_{\mathbf{m}}$ is simple.
\end{proposition}

\begin{proof} 
Consider $V=V_{\mathbf{m}}$ and let $W$  be the induced $\mathfrak{b}$-module. Then every element 
in $W$ can be written as $\sum_{i\geq 0} \mathtt{l}_0^i v_i$ where  $v_i\in V$ and only finitely many of them
are nonzero. Let $X$ be a nonzero submodule of $W$ and $x\in X$ be such that, when written in the above form, 
the maximal  $i$ such that $v_i\neq 0$ is minimal possible, let it be $N$.

Since $V$ is a simple $\mathfrak{n}$-module, using the action of $\mathfrak{n}$ we may assume that $v_N$ is the 
canonical generator $v$ of $V$. We have $\hat{\mathtt{l}}_k\cdot v=m_k v$.

We claim that $y=(\hat{\mathtt{l}}_k-m_k)x$ is nonzero for some $k$ (which reduces $N$ giving a contradiction). 
For this we show that when we write $y$ in the above form, then the coefficient at $\mathtt{l}_0^{N-1}$ will be nonzero.
Clearly, $y$ will not contain any coefficient at $\mathtt{l}_0^N$.

Look at $\mathtt{l}_0^{N-1}V$ (modulo smaller powers of $\mathtt{l}_0$) and choose there an eigenbasis 
for $\hat{\mathtt{l}}_k$ as given by
Lemma~\ref{lemnn1}. There is one eigenvector $\mathtt{l}_0^{N-1}v$ with eigenvalue $m_k$ and the spectrum of
$\hat{\mathtt{l}}_k$ is simple, which means that this eigenvector is not in the image of 
$\hat{\mathtt{l}}_k-m_k$. Hence it is enough to show that $y$ contains a nonzero component at this 
eigenvector coming from level $N$ (and hence this component cannot cancel with anything from level $N-1$ as it 
is not in the image).

Without loss of generality  (namely, by factoring out smaller powers of  $\mathtt{l}_0$) we may assume $N=1$. 
In this case we get that the eigenvector at level $0$ with eigenvalue $m_k+(1-k)z^k$ is  $m_{k+1}v+z^k\mathtt{l}_1v$ 
while the contribution from level $1$ is $km_kv+(k-1)z^{k-1}\mathtt{l}_1v$ (up to a nonzero constant). 
For these two vectors to be linearly dependent we get the equality 
$(k-1)m_{k+1}z^{k-1}=km_kz^k$. This reduces, by recursion, to 
the equations $m_k=(k-1)m_2z^{k-1}$ for $k>2$. The claim follows.
\end{proof}

We will also need the following property of $W_{\mathbf{m}}$:

\begin{lemma}\label{lem95}
Assume that $\mathbf{m}$ satisfies \eqref{conditions}.
\begin{enumerate}[$($a$)$]

\item\label{lem95.1}
The kernels of both
$\hat{\mathtt{l}}_2-m_2$ and $\hat{\mathtt{l}}_3-m_3$ on $W_{\mathbf{m}}$
coincide with $\langle v\rangle$.
\item\label{lem95.2}
The module $W_{\mathbf{m}}$ does not contain any element $x$ such that 
both $(\hat{\mathtt{l}}_2-m_2)\cdot x=v$ and $(\hat{\mathtt{l}}_3-m_3)\cdot x=zv$. 
\end{enumerate}
\end{lemma}

\begin{proof}
Set $y_2:=\frac{1}{m_3-2zm_2}(z\mathtt{l}_0v-\mathtt{l}_1v)$ and 
$y_3:=\frac{z}{2m_4-3zm_3}(z\mathtt{l}_0v-\mathtt{l}_1v)$ 
and note that this is well-defined because of \eqref{conditions}, moreover,  $y_2\neq y_3$ (again by \eqref{conditions}).
A direct computation shows that $(\hat{\mathtt{l}}_2-m_2)\cdot y_2=v$ and
$(\hat{\mathtt{l}}_3-m_3)\cdot y_3=zv$.

Consider the filtration of $W_{\mathbf{m}}$ by $\mathbf{n}$-submodules given by the degree of $\mathtt{l}_0$.
Every layer of this filtration is isomorphic to $V_{\mathbf{m}}$ and from Lemma~\ref{lemnn1} it follows that 
the kernel of $\hat{\mathtt{l}}_2-m_2$  on every layer of this filtration is $1$-dimensional. The computation 
from the previous paragraph implies that each nonzero kernel element at the layer $k$ is sent to a nonzero kernel element 
at  the layer $k-1$. This implies that the kernel of $\hat{\mathtt{l}}_2-m_2$ on $W_{\mathbf{m}}$ coincides
with the kernel of $\hat{\mathtt{l}}_2-m_2$ on the first layer (corresponding to $\mathtt{l}_0^0$). 
The latter one is  exactly  $\langle v\rangle$ by Lemma~\ref{lemnn1}. This proves claim~\eqref{lem95.1} for
$\hat{\mathtt{l}}_2-m_2$ and for $\hat{\mathtt{l}}_3-m_3$ the arguments are similar.

From claim ~\eqref{lem95.1} and the computation above it follows that the equation 
$(\hat{\mathtt{l}}_2-m_2)\cdot x=v$ has solutions $y_2+\alpha v$, $\alpha\in\mathbb{C}$.
None of these equals $y_3$ which implies claim~\eqref{lem95.2}, completing the proof.
\end{proof}

\section{Proof of Theorem~\ref{themmain}}\label{s5}

To prove Theorem~\ref{themmain} we will use a variation of the argument from the proof of
Proposition~\ref{prop:borel}. Assume that $\mathbf{m}\in\mathbb{C}^3$ satisfies \eqref{conditions}.
For $\theta\in\mathbb{C}$, consider the module $\mathrm{Ind}_{z,\theta}(\mathbb{C}_{\mathbf{m}})$.

Denote by $\mathbf{M}$ the set of all infinite vectors $\mathbf{i}=(\dots,i_2,i_1)$ with non-negative 
integer coefficients in which only finitely many coordinates are nonzero. For $\mathbf{i}\in \mathbf{M}$ set
\begin{displaymath}
\mathtt{l}^{\mathbf{i}}:=\dots\mathtt{l}_{-2}^{i_2}\mathtt{l}_{-1}^{i_1}\in U(\mathfrak{V}). 
\end{displaymath}

For $\mathbf{i}\in\mathbf{M}$, the {\em degree} $\mathbf{d}(\mathbf{i})$ is defined as
$\sum_{s\geq 0} i_s$, and the {\em weight} $\mathbf{w}(\mathbf{i})$ is defined as
$\sum_{s\geq 0} si_s$. Let $\mathbf{0}=(\dots,0,0,0)$ and for $s\in\mathbb{Z}_0$ let
$\varepsilon_s$ be the element $(\dots,i_2,i_1)$ such that $i_s=1$ and $i_t=0$ for all $t\neq s$.
Define a total order $\prec$ on $\mathbf{M}$ recursively as follows: $\mathbf{i}\prec\mathbf{j}$ if 
and only if
\begin{itemize}
\item $\mathbf{w}(\mathbf{i})<\mathbf{w}(\mathbf{j})$; or
\item $\mathbf{w}(\mathbf{i})=\mathbf{w}(\mathbf{j})$ and $\mathbf{d}(\mathbf{i})<\mathbf{d}(\mathbf{j})$; or
\item $\mathbf{w}(\mathbf{i})=\mathbf{w}(\mathbf{j})$ and $\mathbf{d}(\mathbf{i})=\mathbf{d}(\mathbf{j})$ and 
\begin{displaymath}
\min\{s\vert i_s\neq 0\}>\min\{s\vert j_s\neq 0\}\text{; or}
\end{displaymath}
\item $\mathbf{w}(\mathbf{i})=\mathbf{w}(\mathbf{j})$ and $\mathbf{d}(\mathbf{i})=\mathbf{d}(\mathbf{j})$ and 
\begin{displaymath}
\min\{s\vert i_s\neq 0\}=\min\{s\vert j_s\neq 0\}=p
\end{displaymath}
and $\mathbf{i}-\varepsilon_p\prec \mathbf{j}-\varepsilon_p$.
\end{itemize}
Clearly, the element $\mathbf{0}$ is the minimum element with respect to this order.

Every element in $\mathrm{Ind}_{z,\theta}(\mathbb{C}_{\mathbf{m}})$ can be uniquely written as a 
sum of $\mathtt{l}^{\mathbf{i}}v_{\mathbf{i}}$, where $\mathbf{i}\in \mathbf{M}$ and 
$v_{\mathbf{i}}\in W_{\mathbf{m}}$ (where only finitely many of the $v_{\mathbf{i}}$'s are non-zero). 
If $x$ is a non-zero element written in this form, then the
{\em support} of $x$ is defined as the following finite set:
\begin{displaymath}
\mathrm{supp}(x):=\{\mathbf{i}:v_{\mathbf{i}}\neq 0\} 
\end{displaymath}
and the {\em maximal term} $\mathbf{t}(x)$ of $x$ is the maximal element in $\mathrm{supp}(x)$ 
(with respect to $\prec$).  Let $M$ be a non-zero submodule in $\mathrm{Ind}_{z,\theta}(\mathbb{C}_{\mathbf{m}})$. 
Denote by  $\mathbf{i}$ the minimal element in the set $\{\mathbf{t}(x):x\in M,x\neq 0\}$. 
To prove Theorem~\ref{themmain} we have to show that $\mathbf{i}=\mathbf{0}$.

\begin{lemma}\label{lem91}
Let $x\in M$, $x\neq 0$, and $u\in\mathfrak{b}$. If $ux\neq 0$, then $\mathbf{t}(ux)\preceq\mathbf{t}(x)$.
\end{lemma}

\begin{proof}
By linearity, it is enough to prove the claim for $u=\mathtt{l}_k$ and $x=\mathtt{l}^{\mathbf{i}}v_{\mathbf{i}}$.
We have $\mathtt{l}_k\cdot  \mathtt{l}^{\mathbf{i}}v_{\mathbf{i}}=
\mathtt{l}^{\mathbf{i}}\mathtt{l}_k\cdot v_{\mathbf{i}}+
[\mathtt{l}_k,\mathtt{l}^{\mathbf{i}}]v_{\mathbf{i}}$.
From the definition of $\prec$ it follows that, moving $\mathtt{l}_k$ to the right in
$[\mathtt{l}_k,\mathtt{l}^{\mathbf{i}}]$, we get a linear combination of some $\mathtt{l}^{\mathbf{j}}$
(with possible element from $\mathfrak{b}$ on the right)
where $\mathbf{j}\preceq\mathbf{i}$. The claim follows.
\end{proof}

Assume that $\mathbf{i}\neq \mathbf{0}$ and let $x\in M$ be some non-zero element with maximal term $\mathbf{i}$. 
Since $W_{\mathbf{m}}$ is a simple module by Proposition~\ref{prop:borel}, without loss of generality we
may assume $v_{\mathbf{i}}=v$ (the canonical generator of $V_{\mathbf{m}}$). For $k\geq 2$ consider the 
element $y_k:=(\hat{\mathtt{l}}_k-m_k)\cdot x$ and we have either $y_k= 0$ or 
$\mathbf{i}\not\in\mathrm{supp}(y_k)$. Therefore $y_k\neq 0$ implies $\mathbf{t}(y_k)\prec \mathbf{i}$
by Lemma~\ref{lem91}, which gives a contradiction completing the proof of Theorem~\ref{themmain}.
In what follows we show that either $y_2\neq 0$ or $y_3\neq 0$.

Let $p:=\min\{s:i_s\neq 0\}$ and consider first the case $p>1$. We claim that 
$\mathbf{j}:=\mathbf{i}-\varepsilon_p+\varepsilon_{p-1}$ belongs to either 
$\mathrm{supp}(y_2)$ or $\mathrm{supp}(y_3)$, which implies that at least one of these elements is
nonzero. Assume that this is not the case and let $k\in\{2,3\}$. Let $\mathbf{i}'\in \mathrm{supp}(x)$ be different
from $\mathbf{i}$. Using the definition of $\prec$, it is easy to see that, writing 
$[\hat{\mathtt{l}}_k-m_k,\mathtt{l}^{\mathbf{i}'}]$  as a linear combination 
of $\mathtt{l}^{\mathbf{s}}u_{\mathbf{s}}$, where $u_{\mathbf{s}}\in\mathfrak{b}$, all $\mathbf{s}$ which will appear 
with nonzero $u_{\mathbf{s}}$ satisfy $\mathbf{s}\prec \mathbf{j}$.

At the same time, $[\hat{\mathtt{l}}_k-m_k,\mathtt{l}^{\mathbf{i}}]$ contributes with
$-i_p(p+1)z^{k-1}\mathtt{l}^{\mathbf{j}}$. If $y_k=0$, then to cancel this contribution, we thus must have 
$\mathbf{j}\in \mathrm{supp}(x)$ and $v_{\mathbf{j}}\in W_{\mathbf{m}}$ should satisfy 
$(\hat{\mathtt{l}}_k-m_k)\cdot v_{\mathbf{j}}=i_p(p+1)z^{k-1}v$. However, such $v_{\mathbf{j}}$ does not exist
by Lemma~\ref{lem95}\eqref{lem95.2}, which completes the proof in the case $p>1$.

Assume now that $p=1$. Our argument will be similar to the one we used above, however, it will require 
a computationally more complicated analogue of Lemma~\ref{lem95}. We claim that 
$\mathbf{j}:=\mathbf{i}-\varepsilon_1$ belongs to either $\mathrm{supp}(y_2)$ or $\mathrm{supp}(y_3)$.
Assume that this is not the case. Consider the coefficient at $\mathtt{l}^{\mathbf{j}}$ in both
$y_2$ and $y_3$. Similarly to the above, the only contribution to this coefficient comes from
$[\hat{\mathtt{l}}_k-m_k,\mathtt{l}^{\mathbf{i}}]$ and from
$(\hat{\mathtt{l}}_k-m_k)\cdot v_{\mathbf{j}}$. So, if this coefficient is zero, these two contributions
should cancel each other for both $k=2$ and $k=3$.

One checks that the contribution from $[\hat{\mathtt{l}}_2-m_2,\mathtt{l}^{\mathbf{i}}]$ equals exactly
$i_1(-3\mathtt{l}_1+2z\mathtt{l}_0 +z(i_1-1))v$ and the contribution from 
$[\hat{\mathtt{l}}_3-m_3,\mathtt{l}^{\mathbf{i}}]$ equals $i_1(-4m_2-4z\mathtt{l}_1+2z^2\mathtt{l}_0+z^2(i_1-1))v$.
Now the claim follows, similarly to the above, from the following lemma:

\begin{lemma}\label{lem97}
Let $t\in\mathbb{N}$.
If $\mathbf{m}$ satisfies \eqref{conditions}, then $W_{\mathbf{m}}$ does not contain any element $x$ such that
\begin{displaymath}
\left\{
\begin{array}{lcl}
(\hat{\mathtt{l}}_2-m_2)\cdot x&=& t(-3\mathtt{l}_1+2z\mathtt{l}_0+z(t-1))v,\\
(\hat{\mathtt{l}}_3-m_3) \cdot x&= &t(-4m_2-4z\mathtt{l}_1+2z^2\mathtt{l}_0+z^2(t-1))v.
\end{array}
\right.
\end{displaymath}
\end{lemma}

\begin{proof}
By Lemma~\ref{lem95}\eqref{lem95.1}, each of the above equations, if solvable, has a $1$-dimensional set of solutions
(which differ by scalar multiples of $v$). It is straightforward to check that the element
{\small
\begin{multline*}
\frac{t}{2z^2m_2-zm_3}\big(-z^3\mathtt{l}_0^2-\\-
\frac{2z^5(1+t)m_2-z^4(4+t)m_3+2z^2m_2m_3+2z^3m_4-zm_3^2}{2z^2m_2-zm_3}\mathtt{l}_0
+\\+2z^2\mathtt{l}_0\mathtt{l}_1+\\
+\frac{2z^4tm_2+4z^2m_2^2-z^3(3+t)m_3-2zm_2m_3+2z^2m_4}{2z^2m_2-zm_3}\mathtt{l}_1
-z\mathtt{l}_1^2\big)\cdot v
\end{multline*}
}\noindent
solves the first equation and that the element
{\small
\begin{multline*}
\frac{t}{3zm_3-2m_4}\big(-z^3\mathtt{l}_0^2+(4zm_2-\frac{2}{z}m_4-tz^3)\mathtt{l}_0
+2z^2\mathtt{l}_0\mathtt{l}_1+\\
+(tz^2-z^2-4m_2+\frac{3}{z}m_3)\mathtt{l}_1-z\mathtt{l}_1^2\big)\cdot v
\end{multline*}
}\noindent
solves the second equation. Comparing the coefficient at $\mathtt{l}_0^2$, from \eqref{conditions} it follows 
that these two elements are different, moreover, they do not differ by a scalar multiple of $v$. The claim of 
the lemma follows.
\end{proof}

Both solutions in the proof of Lemma~\ref{lem97} were found and also Gr{\"o}bner basis computations in the 
proof of Proposition~\ref{prop2} were performed using MAPLE, but it is straightforward to check that they are correct.
Note that from the definition of our modules it follows immediately that different values of parameters
give rise to non-isomorphic simple modules. It is also easy to check (looking at the action of $\mathfrak{a}_z$)
that our modules are not isomorphic to any of the previously known simple Virasoro modules.

\vspace{0.2cm}

\noindent
V.M.: Department of Mathematics, Uppsala University, 
Box 480, SE-751 06, Uppsala, Sweden; e-mail: {\tt mazor\symbol{64}math.uu.se}
\vspace{0.5cm}

\noindent E.W.: Department of Mathematics, Ithaca College, Williams Hall
Ithaca, NY 14850, USA. e-mail:  {\tt ewiesner\symbol{64}ithaca.edu}

\end{document}